\documentclass[11pt]{amsart}
\usepackage{amsmath,amssymb,mathrsfs}
\def\R{{\mathbb R}}

\newtheorem{thm}{Theorem}[section]

\newtheorem{lem}[thm]{Lemma}

\theoremstyle{definition}
\newtheorem{de}[thm]{Definition}
\theoremstyle{remark}
\newtheorem{rem}[thm]{Remark}
\numberwithin{equation}{section}
\newcommand{\rmd}{{\rm d}}

\allowdisplaybreaks

\begin{document}

\title[automorphic solutions for SDE]{Square-mean almost
automorphic solutions for some stochastic differential equations}

\author{Miaomiao Fu}
\address{M. Fu: 1) College of Mathematics,
Jilin University, Changchun 130012, P. R. China \\
2) School of Mathematics, Changchun Normal College, Changchun
130032, P. R. China} \email{mmfucaathy@yahoo.com.cn}

\thanks{The second author is partially
supported by NSFC Grant 10801059, SRFDP Grant 20070183053, the 985
Program of Jilin University, and the science research fund at Jilin
University.}

\author{Zhenxin Liu}
\address{Z. Liu: College of Mathematics,
Jilin University, Changchun 130012, P. R. China}
\email{zxliu@jlu.edu.cn}

\date{}



\begin{abstract}
The concept of square-mean almost automorphy for stochastic
processes is introduced. The existence and uniqueness of square-mean
almost automorphic solutions to  some linear and non-linear
stochastic differential equations are established provided the
coefficients satisfy some conditions. The asymptotic stability of
the unique square-mean almost automorphic solution in square-mean
sense is discussed.
\end{abstract}

\maketitle

\section{Introduction}
\setcounter{equation}{0} The concept of almost automorphy is a
generalization of almost periodicity. It is introduced by Bochner
\cite{Boc} in relation to some aspects of differential geometry.
Almost automorphic functions are characterized by the following
property. Let $f$ be a continuous function, given any sequence of
real numbers $\{s'_n\}$, we can extract a subsequence $\{s_n\}$ such
that for some function $g$
\begin{equation*}
g(t)=\lim_{n\rightarrow \infty}f(t+s_n)\quad\hbox{and}\quad
\lim_{n\rightarrow \infty}g(t-s_n)=f(t)
\end{equation*}
for each $t\in \mathbb{R}$. Almost automorphy has been studied by
many authors, see Veech \cite{Vee65,Vee77} for classical exposition;
see Johnson \cite{Joh}, Shen and Yi \cite{SY}, N'Gu\'er\'ekata
\cite{Gue01} for recent development, among others.

Recently, some authors study the almost periodic or pseudo almost
periodic solutions to stochastic differential equations, see
\cite{AT,BD,BD1,Pra95,Hal,Tud,Tudd}, among others. In this paper, we
go one step further by introducing the concept of square-mean almost
automorphic stochastic processes. Under some conditions of
coefficients, we establish existence and uniqueness of square-mean
almost automorphic solutions for some stochastic differential
equations.

The paper is organized as follows. In section 2, we introduce the
notion of square-mean almost automorphic processes and study some of
their basic properties. In sections 3 and 4, given some suitable
conditions, we prove the existence and uniqueness of square-mean
almost automorphic mild solutions to some linear and non-linear
stochastic differential equations, respectively. In section 5, we
discuss the asymptotic stability property of the unique square-mean
almost automorphic solution in square-mean sense.

\section{Square-mean almost automorphic processes}
\setcounter{equation}{0}

Throughout this paper, we assume that $(\mathbb{H}, \|\cdot\|)$ is a
real separable Hilbert space, $(\Omega, \mathcal{F}, \mathbf{P})$ is
a probability space, and $\mathcal{L}^2 (\mathbf{P},\mathbb{H})$
stands for the space of all $\mathbb{H}$-valued random variables $x$
such that
\begin{eqnarray*}
\mathbf{E}\|x\|^2=\int_{\Omega}\|x\|^2d\mathbf{P}<\infty.
\end{eqnarray*}
For $x\in\mathcal{L}^2 (\mathbf{P},\mathbb{H})$, let
\[
\|x\|_2:=\left(\int_{\Omega}\|x\|^2d\mathbf{P}\right)^{1/2}.
\]
Then it is routine to check that $\mathcal{L}^2
(\mathbf{P},\mathbb{H})$ is a Hilbert space equipped with the norm
$\|\cdot\|_2$.

\begin{de}
A stochastic process $X: \mathbb{R} \rightarrow
\mathcal{L}^2(\mathbf{P}, \mathbb{H})$ is said to be {\em
stochastically continuous} if
\begin{eqnarray*}
\lim_{t\rightarrow s} \mathbf{E}\|X(t)-X(s)\|^2
 = 0.
\end{eqnarray*}
\end{de}

\begin{de}
A stochastically continuous stochastic process $x:\mathbb{R}
\rightarrow \mathcal{L}^2(\mathbf{P}, \mathbb{H})$ is said to be
{\em square-mean almost automorphic}  if every sequence of real
numbers $\{s'_n\}$ has a subsequence $\{s_n\}$ such that for some
stochastic process $y:\mathbb{R} \rightarrow
\mathcal{L}^2(\mathbf{P}, \mathbb{H})$
\begin{equation*}
\lim_{n\rightarrow \infty}\mathbf{E} \|x(t+s_n) - y(t)\|^2=0
\quad\hbox{and}\quad \lim_{n\rightarrow \infty}\mathbf{E} \|y(t-s_n)
- x(t)\|^2=0
\end{equation*}
holds for each $t\in\mathbb R$. The collection of all square-mean
almost automorphic stochastic processes $x:\mathbb{R}\rightarrow
\mathcal{L}^2(\mathbf{P},\mathbb{H})$ is denoted by
$AA(\mathbb{R};\mathcal{L}^2(\mathbf{P},\mathbb{H}))$.
\end{de}

In the following lemma, we list some basic properties of square-mean
almost automorphic stochastic processes.
\begin{lem}\label{lem}
If $x$, $x_1$ and $x_2$ are all square-mean almost automorphic
stochastic processes, then
\begin{enumerate}
  \item $x_1+x_2$ is square-mean almost automorphic.
  \item $\lambda x$ is square-mean almost automorphic for every scalar
         $\lambda$.
  \item There exists a constant $M>0$ such that $\sup_{t\in
\mathbb{R}}\|x(t)\|_2 \leq M$. That is, $x$ is bounded in
$\mathcal{L}^2(\mathbf{P}, \mathbb{H})$.
\end{enumerate}
\end{lem}

\begin{proof}
Since statements (1) and (2) are obvious, we only prove (3). If
$\sup_{t\in \mathbb{R}}\|x(t)\|_2=\infty$, then there exists a
sequence of  real numbers $\{s'_n\}$ such that
\begin{eqnarray*}
\lim_{n\rightarrow \infty}\|x(s'_n)\|_2^2=\lim_{n\rightarrow
\infty}\mathbf{E}\|x(s'_n)\|^2=\infty.
\end{eqnarray*}
Since $x\in AA(\mathbb{R};\mathcal{L}^2(\mathbf{P},\mathbb{H}))$,
there exists a subsequence $\{s_n\}\subset\{s'_n\}$ and a stochastic
process $y:\mathbb{R} \rightarrow \mathcal{L}^2(\mathbf{P},
\mathbb{H})$ such that
\begin{equation}\label{la}
\lim_{n\rightarrow
\infty}\mathbf{E}\|x(t+s_n)-y(t)\|^2=0\quad\hbox{for
all}~~t\in\mathbb R.
\end{equation}
In particular, when $t=0$ in \eqref{la}, it follows that
\begin{eqnarray*}
\lim_{n\rightarrow \infty}\mathbf{E}\|x(s_n)\|^2<\infty,
\end{eqnarray*}
a contradiction.
\end{proof}

\begin{thm}
$AA(\mathbb{R};\mathcal{L}^2(\mathbf{P},\mathbb{H}))$ is a Banach
space when it is equipped with the norm
\begin{eqnarray*}
\|x\|_{\infty}:=\sup_{t\in\mathbb R}\|x(t)\|_2=\sup_{t\in
\mathbb{R}}(E\|x(t)\|^2)^{\frac{1}{2}},
\end{eqnarray*}
for $x\in AA(\mathbb{R};\mathcal{L}^2(\mathbf{P},\mathbb{H}))$.
\end{thm}

\begin{proof}
By lemma \ref{lem},
$AA(\mathbb{R};\mathcal{L}^2(\mathbf{P},\mathbb{H}))$ is a vector
space, then it is easy to verify that $\|\cdot\|_\infty$ is a norm
on $AA(\mathbb{R};\mathcal{L}^2(\mathbf{P},\mathbb{H}))$. We only
need to show that
$AA(\mathbb{R};\mathcal{L}^2(\mathbf{P},\mathbb{H}))$ is complete
with respect to the norm $\|\cdot\|_\infty$. To this end, assume
that $\{x_n\}\subset
AA(\mathbb{R};\mathcal{L}^2(\mathbf{P},\mathbb{H}))$ is a Cauchy
sequence with respect to $\|\cdot\|_\infty$ and that $x$ is the
pointwise limit of $x_n$ with respect to $\|\cdot\|_2$, i.e.
\begin{equation}\label{eq}
\lim_{n\rightarrow \infty}\|x_n(t)-x(t)\|_2=0\quad\hbox{for each}~~
t\in \mathbb{R}.
\end{equation}
Note that this limit $x$ always exists by the completeness of
$\mathcal{L}^2(\mathbf{P},\mathbb{H})$ with respect to
$\|\cdot\|_2$. Since $\{x_n\}$ is Cauchy with respect to
$\|\cdot\|_\infty$, the convergence in \eqref{eq} is actually
uniform for $t\in\mathbb R$. We need to show that $x\in
AA(\mathbb{R};\mathcal{L}^2(\mathbf{P},\mathbb{H}))$.

Firstly, we verify that $x$ is stochastically continuous. In fact,
by
\begin{eqnarray*}
x(t+\bigtriangleup t)-x(t)=x(t+\bigtriangleup
t)-x_n(t+\bigtriangleup t)+x_n(t+\bigtriangleup
t)-x_n(t)+x_n(t)-x(t),
\end{eqnarray*}
it follows that
\begin{eqnarray*}
&&\mathbf{E}\|x(t+\bigtriangleup t)-x(t)\|^2\\
&\leq& 3\mathbf{E}\|x(t+\bigtriangleup t)-x_n(t+\bigtriangleup
t)\|^2+3\mathbf{E}\|x_n(t+\bigtriangleup
t)-x_n(t)\|^2+3\mathbf{E}\|x(t)-x_n(t)\|^2.
\end{eqnarray*}
By the uniform convergence of $x_n$ to $x$ with respect to
$\|\cdot\|_2$ and the stochastic continuity of $x_n$, the stochastic
continuity of $x$ follows.

Next, we prove that $x$ is square-mean almost automorphic. Let
$\{s'_n\}$ be an arbitrary sequence of real numbers, then by
standard diagonal progress, we can extract a subsequence
$\{s_n\}\subset\{s'_n\}$ such that for stochastic processes
$y_i:\mathbb{R} \rightarrow \mathcal{L}^2(\mathbf{P}, \mathbb{H})$
\begin{equation}\label{equ}
\lim_{n\rightarrow\infty} \mathbf{E}\|x_i(t+s_n)-y_i(t)\|^2=0
\end{equation}
for each $t\in \mathbb{R}$ and $i=1,2,\cdots$.

We observe that, for each $t\in\mathbb R$, the sequence of
$\{y_i(t)\}$ is a Cauchy sequence in
$\mathcal{L}^2(\mathbf{P},\mathbb{H})$. Indeed, if we write
\begin{eqnarray*}
y_i(t)-y_j(t)=y_i(t)-x_i(t+s_n)+x_i(t+s_n)-x_j(t+s_n)+x_j(t+s_n)-y_j(t),
\end{eqnarray*}
then we get
\begin{eqnarray*}
&&\mathbf{E}\|y_i(t)-y_j(t)\|^2\\&\le&3\mathbf{E}\|y_i(t)-x_i(t+s_n)\|^2
+3\mathbf{E}\|x_i(t+s_n)-x_j(t+s_n)\|^2+3\mathbf{E}\|x_j(t+s_n)-y_j(t)\|^2.
\end{eqnarray*}
By \eqref{eq} and \eqref{equ}, the sequence of $\{y_i(t)\}$ is
Cauchy.

Using the completeness of the space
$\mathcal{L}^2(\mathbf{P},\mathbb{H})$, we denote by $y(t)$ the
pointwise limit of $\{y_i(t)\}$. Let us prove now that
\begin{equation*}
\lim_{n\rightarrow\infty}
\mathbf{E}\|x(t+s_n)-y(t)\|^2=0\quad\hbox{and}\quad
\lim_{n\rightarrow\infty} \mathbf{E}\|y(t-s_n)-x(t)\|^2=0
\end{equation*}
for each $t\in\mathbb R$. Indeed, for each $i=1,2,\cdots$, we have
\begin{eqnarray*}
&&\mathbf{E}\|x(t+s_n)-y(t)\|^2\\&\le&3\mathbf{E}\|x(t+s_n)-x_i(t+s_n)\|^2
+3\mathbf{E}\|x_i(t+s_n)-y_i(t)\|^2+3\mathbf{E}\|y_i(t)-y(t)\|^2.
\end{eqnarray*}
By \eqref{eq} and \eqref{equ},
\begin{eqnarray*}
\lim_{n\rightarrow\infty}\mathbf{E}\|x(t+s_n)-y(t)\|^2=0.
\end{eqnarray*}
for each $t\in \mathbb{R}$.

We can use the same step to prove that
\begin{eqnarray*}
\lim_{n\rightarrow\infty}\mathbf{E}\|y(t-s_n)-x(t)\|^2=0
\end{eqnarray*}
for each $t\in \mathbb{R}$. That is, $x(t)$ is square-mean almost
automorphic. The proof is complete.
\end{proof}

\begin{de}
A function $f: \mathbb{R}\times \mathcal{L}^2(\mathbf{P},
\mathbb{H}) \rightarrow \mathcal{L}^2(\mathbf{P}, \mathbb{H})$,
$(t,x)\mapsto f(t,x)$, which is jointly continuous, is said to be
{\em square-mean almost automorphic} in $t\in \mathbb{R}$ for each
$x\in \mathcal{L}^2(\mathbf{P}, \mathbb{H})$ if for every sequence
of real numbers $\{s'_n\}$, there exists a subsequence $\{s_n\}$
such that for some function $\widetilde{f}$
\begin{equation*}
\lim_{n\rightarrow \infty}\mathbf{E} \|f(t+s_n,x) -
\widetilde{f}(t,x)\|^2 =0 \quad\hbox{and}\quad \lim_{n\rightarrow
\infty}\mathbf{E} \|\widetilde{f}(t-s_n,x)-f(t,x)\|^2 =0
\end{equation*}
for each $t\in \mathbb{R}$ and each $x\in \mathcal{L}^2(\mathbf{P},
\mathbb{H})$.
\end{de}

\begin{thm}\label{th}
Let $f: \mathbb{R}\times \mathcal{L}^2(\mathbf{P}, \mathbb{H})
\rightarrow \mathcal{L}^2(\mathbf{P}, \mathbb{H})$, $(t,x)\mapsto
f(t,x)$ be square-mean almost automorphic in $t\in \mathbb{R}$ for
each $x\in \mathcal{L}^2(\mathbf{P}, \mathbb{H})$, and assume that
$f$ satisfies Lipschitz condition in the follow sense:
\begin{eqnarray*}
\mathbf{E}\|f(t,x)-f(t,y)\|^2\leq L\mathbf{E}\|x-y\|^2
\end{eqnarray*}
for all $x,y\in \mathcal{L}^2(\mathbf{P}, \mathbb{H})$, and for each
$t\in \mathbb{R}$, where $L>0$ is independent of $t$. Then for any
square-mean almost automorphic process $x:\mathbb{R}\rightarrow
\mathcal{L}^2(\mathbf{P}, \mathbb{H})$,  the stochastic process
$F:\mathbb{R}\rightarrow \mathcal{L}^2(\mathbf{P}, \mathbb{H})$
given by $F(t):=f(t,x(t))$ is square-mean almost automorphic.
\end{thm}
\begin{proof}
Let $\{s'_n\}$ be a sequence of real numbers. By almost automorphy
of ${f}$ and $x$, we can extract a subsequence $\{s_n\}$ of
$\{s'_n\}$ such that for some function $\widetilde{f}$ and for each
$t\in \mathbb{R}$ and $x\in \mathcal{L}^2(\mathbf{P}, \mathbb{H})$,
\begin{equation}\label{eq1}
\lim_{n \rightarrow
\infty}\mathbf{E}\|f(t+s_n,x)-\widetilde{f}(t,x)\|^2=0,
\end{equation}
and for some function $y$ and for each $t\in \mathbb{R}$,
\begin{equation}\label{eq2}
\lim_{n \rightarrow \infty}\mathbf{E}\|x(t+s_n)-y(t)\|^2=0.
\end{equation}

Let us consider the function $\widetilde{F}:\mathbb{R}\rightarrow
\mathcal{L}^2(\mathbf{P}, \mathbb{H})$ defined by
$\widetilde{F}(t):=\widetilde{f}(t,y(t))$, $t\in \mathbb{R}$. Note
that
\begin{eqnarray*}
F(t+s_n)-\widetilde{F}(t)=f(t+s_n,x(t+s_n))-f(t+s_n,y(t))+f(t+s_n,y(t))-\widetilde{f}(t,y(t)),
\end{eqnarray*}
so we have
\begin{eqnarray*}
&&\mathbf{E}\|F(t+s_n)-\widetilde{F}(t)\|^2\\
&\leq&2\mathbf{E}\|f(t+s_n,x(t+s_n))-f(t+s_n,y(t))\|^2+2\mathbf{E}\|f(t+s_n,y(t))-\widetilde{f}(t,y(t))\|^2\\
&\leq&2L\mathbf{E}\|x(t+s_n)-y(t)\|^2+2\mathbf{E}\|f(t+s_n,y(t))-\widetilde{f}(t,y(t))\|^2.
\end{eqnarray*}
We can deduce from \eqref{eq1} and \eqref{eq2} that
\begin{eqnarray*}
\lim_{n\rightarrow
\infty}\mathbf{E}\|F(t+s_n)-\widetilde{F}(t)\|^2=0,~~~~\mbox{for
each}~~~ t\in \mathbb{R}.
\end{eqnarray*}

Similarly we can prove that $\lim_{n\rightarrow
\infty}\mathbf{E}\|\widetilde{F}(t-s_n)-F(t)\|^2=0$ for each $t\in
\mathbb{R}$, which proves the square-mean almost automorphy of
$F(t)$.
\end{proof}

\section{The linear stochastic differential equations}
\setcounter{equation}{0} Consider the following linear stochastic
differential equation
\begin{equation}\label{dt}
\rmd x(t)=Ax(t)\rmd t+f(t)\rmd t+g(t)\rmd W(t),~~~~t \in \mathbb{R},
\end{equation}
where $A$ is an infinitesimal generator which generates a
$\mathcal{C}_0$-semigroup $(T(t)_{t\geq 0})$, such that
\begin{equation}\label{exp}
\|T(t)\| \leq Ke^{-\omega t},~~~~\mbox{for all}~~~ t \geq 0
\end{equation}
with $K>0$, $\omega>0$. And  $f:\mathbb{R} \rightarrow
\mathcal{L}^2(\mathbf{P}, \mathbb{H})$, $g:\mathbb{R} \rightarrow
\mathcal{L}^2(\mathbf{P}, \mathbb{H})$ are stochastic processes,
$W(t)$ is a two-sided standard one-dimensional Brown motion defined
on the filtered probability space $(\Omega,
\mathcal{F},\mathbf{P},\mathcal{F}_t)$, where $\mathcal
F_t=\sigma\{W(u)-W(v);u,v\le t\}$.

\begin{de}
An $\mathcal{F}_t$-progressively measurable process $\{x(t)\}_{t\in
\mathbb{R}}$ is called a {\em mild solution} of (\ref{dt}) if it
satisfies the corresponding stochastic integral equation
\begin{eqnarray*}
x(t)=T(t-a)x(a)+\int_a^tT(t-s)f(s)\rmd s+\int_a^tT(t-s)g(s)\rmd
W(s),
\end{eqnarray*}
for all $t\geq a$ and each $a\in \mathbb{R}$.
\end{de}
\begin{thm}\label{th1}
Given $f,g\in AA(\mathbb{R}; \mathcal{L}^2(\mathbf{P},\mathbb{H}))$,
(\ref{dt}) has a unique square-mean almost automorphic mild
solution.
\end{thm}
\begin{proof}
It is well known (see \cite{Arn}) that for given $a\in \mathbb{R}$
and given initial value $x_a$ at `time' $a$, the process
\begin{equation}\label{so}
x(t)=T(t-a)x_a+\int_a^tT(t-s)f(s)\rmd s+\int_a^tT(t-s)g(s)\rmd
W(s),\quad t\ge a
\end{equation}
is the unique mild solution to \eqref{dt} with the initial value
condition $x(a)=x_a$. So to prove the existence of square-mean
almost automorphic mild solution, we need to find an initial value
$x_a$ such that the stochastic process given by \eqref{so} is
square-mean almost automorphic.

Let $\{s'_n\}$ be an arbitrary sequence of real numbers. Since $f$
and $g$ are square-mean almost automorphic, there exists a
subsequence $\{s_n\}$ of $\{s'_n\}$  such that for certain
stochastic processes $\widetilde{f}$ and $\widetilde{g}$
\begin{equation*}
\lim_{n\rightarrow
\infty}\mathbf{E}\|f(t+s_n)-\widetilde{f}(t)\|^2=0,\qquad
\lim_{n\rightarrow
\infty}\mathbf{E}\|\widetilde{f}(t-s_n)-f(t)\|^2=0
\end{equation*}
and
\begin{equation*}
\lim_{n\rightarrow
\infty}\mathbf{E}\|g(t+s_n)-\widetilde{g}(t)\|^2=0,\qquad
\lim_{n\rightarrow
\infty}\mathbf{E}\|\widetilde{g}(t-s_n)-g(t)\|^2=0
\end{equation*}
hold for each $t\in \mathbb{R}$.

Now we consider $u(t):=\int_{-\infty}^tT(t-s)f(s)\rmd s$, defined as
$\lim_{r\rightarrow -\infty}\int_r^tT(t-s)f(s)\rmd s$. From
\cite[theorem 3.1]{Gue04}, we know that $\int_r^tT(t-s)f(s)\rmd s$
exists for each $r<t$. Moreover, if we let
$\widetilde{u}(t):=\int_{-\infty}^tT(t-s)\widetilde{f}(s)\rmd s$, we
have
\begin{eqnarray*}
u(t+s_n)\rightarrow \widetilde{u}(t)\quad\hbox{and}\quad
\widetilde{u}(t-s_n)\rightarrow u(t)\quad\hbox{in}~~~\mathcal
L^2(\mathbf P,\mathbb H),\quad\mbox{as}~~~ n\rightarrow \infty
\end{eqnarray*}
for each $t\in \mathbb{R}$. This indicates that $u$ is square-mean
almost automorphic. Note that $u(a)=\int_{-\infty}^aT(a-s)f(s)\rmd
s$. If $t\geq a$, then
\begin{eqnarray}
\int_a^tT(t-s)f(s)\rmd s&=&\int_{-\infty}^tT(t-s)f(s)\rmd s-\int_{-\infty}^aT(t-s)f(s)\rmd s\nonumber\\
&=&u(t)-T(t-a)u(a),\nonumber
\end{eqnarray}
i.e.,
\[
u(t)=T(t-a)x(a)+\int_a^tT(t-s)f(s)\rmd s.
\]
 If we choose initial value $x_a=u(a)$, then the process $x(t)$ given
by \eqref{so} is square-mean almost automorphic. In fact, denote
\[
\hat{u}(t):=T(t-a)x(a)+\int_a^tT(t-s)\widetilde{f}(s)\rmd s
\]
and
\begin{eqnarray*}
\tilde{x}(t):=T(t-a)x(a)+\int_a^tT(t-s)\widetilde{f}(s)\rmd
s+\int_a^tT(t-s)\widetilde{g}(s)\rmd W(s).
\end{eqnarray*}
Then, for each $t\in \mathbb{R}$, we have
\begin{eqnarray*}
&&\lim_{n\rightarrow \infty} \mathbf{E}\|x(t+s_n)-\tilde{x}(t)\|^2\\
&=&\lim_{n\rightarrow \infty}
\mathbf{E}\|T(t+s_n-a)x(a)+\int_a^{t+s_n}
T(t+s_n-s)f(s)\rmd s+\int_a^{t+s_n} T(t+s_n-s)g(s)\rmd W(s)\\
&&-T(t-a)x(a)-\int_a^t
T(t-s)\widetilde{f}(s)\rmd s-\int_a^t T(t-s)\widetilde{g}(s)\rmd W(s)\|^2\\
&=&\lim_{n\rightarrow \infty}
\mathbf{E}\|u(t+s_n)-\hat{u}(t)+\left(\int_{-\infty}^{t+s_n}-\int_{-\infty}^a\right)
T(t+s_n-s)g(s)\rmd W(s)\\
&&\quad -\left(\int_{-\infty}^t-\int_{-\infty}^a\right)
T(t-s)\widetilde{g}(s)\rmd W(s)\|^2.
\end{eqnarray*}
Let $\tilde{W}(\sigma):=W(\sigma+s_n)-W(s_n)$ for each
$\sigma\in\mathbb R$. Note that $\tilde{W}$ is also a Wiener process
and has the same distribution as $W$. Hence by making a changing
$\sigma=s-s_n$ and by the Ito's isometry property of stochastic
integral (see \cite[page 29]{Oks}, for example), we have
\begin{eqnarray*}
&&\mathbf{E}\|x(t+s_n)-\tilde{x}(t)\|^2\\
&\le&3\mathbf{E}\|u(t+s_n)-\hat{u}(t)\|^2 +3\mathbf{E}\left\|
\int_{-\infty}^tT(t-\sigma)[g(\sigma+s_n)-\widetilde{g}(\sigma)]\rmd\tilde{W}(\sigma)\right\|^2\\
&&+3\mathbf{E}\left\|\int_{-\infty}^aT(t-\sigma)[g(\sigma+s_n)-
\widetilde{g}(\sigma)]\rmd\tilde{W}(\sigma)\right\|^2\\
&\leq& 3\mathbf{E}\|u(t+s_n)-\hat{u}(t)\|^2 +3
\left[\int_{-\infty}^t\|T(t-\sigma)\|^2\mathbf{E}\|g(\sigma+s_n)-\widetilde{g}(\sigma)\|^2
\rmd\sigma\right]\\
&&+3\left[\int_{-\infty}^a\|T(t-\sigma)\|^2\mathbf{E}\|g(\sigma+s_n)-
\widetilde{g}(\sigma)\|^2\rmd\sigma\right].
\end{eqnarray*}
Noting that $g$ is square-mean almost automorphic and by the
exponential dissipation property (\ref{exp}) of $T(t)$ , we
immediately obtain that
\begin{eqnarray*}
\lim_{n\rightarrow \infty} \mathbf{E}\|x(t+s_n)-\tilde{x}(t)\|^2=0.
\end{eqnarray*}
And we can show in a similar way that
\begin{eqnarray*}
\lim_{n\rightarrow \infty} \mathbf{E}\|\tilde{x}(t-s_n)-x(t)\|^2=0
\end{eqnarray*}
for each $t\in \mathbb{R}$. Up to now, the existence is proved.

We finally prove the uniqueness of the square-mean almost
automorphic solution of \eqref{dt}. Assume that $x(t)$ and $y(t)$
are both square-mean almost automorphic solutions of \eqref{dt} with
different initial value $x(a)$ and $y(a)$ at `time' $a$. That is,
for $t\geq a$,
\begin{eqnarray*}
x(t)&=&T(t-a)x(a)+\int_a^tT(t-s)f(s)\rmd s+\int_a^tT(t-s)g(s)\rmd
W(s),\\
y(t)&=&T(t-a)y(a)+\int_a^tT(t-s)f(s)\rmd s+\int_a^tT(t-s)g(s)\rmd
W(s),
\end{eqnarray*}
and $x(a)\neq y(a)$. Let $z(t)=x(t)-y(t)$. Then $z(t)$ satisfies the
equation
\begin{eqnarray*}
\rmd z(t)=Az(t)\rmd t,\quad t\ge a
\end{eqnarray*}
with initial condition $z(a)=x(a)-y(a)$. Hence
\begin{eqnarray*}
z(t)=T(t-a)z(a)
\end{eqnarray*}
and
\begin{equation*}
\|z(t)\| \leq Ke^{-\omega (t-a)}\|z(a)\|,\quad\mbox{for all}~~~ t
\geq a.
\end{equation*}
So $z(t) \rightarrow 0$, as $t \rightarrow +\infty$. Since $z(t)\in
AA(\mathbb{R};\mathcal{L}^2(\mathbf{P},\mathbb{H}))$, for any
sequence of real numbers $\{s'_n\}$, there exists a subsequence
$\{s_n\}$ of $\{s'_n\}$ such that for some progress $\tilde z(t)$,
\begin{eqnarray}\label{z}
\lim_{n\rightarrow \infty} z(t+s_n)=\tilde{z}(t)\quad\hbox{and}\quad
\lim_{n\rightarrow \infty} \tilde{z}(t-s_n)=z(t)
\end{eqnarray}
for each $t\in \mathbb{R}$. In particular, if
$\lim_{n\to\infty}s'_n=\infty$, then $\tilde{z}(t)\equiv 0$ by the
first equality of \eqref{z}. Hence $z(t)\equiv 0$ by the second
equality of \eqref{z}, so we must have $x(a)=y(a)$, a contradiction.
The proof is complete.
\end{proof}

\section{The non-linear stochastic differential equations}
\setcounter{equation}{0}

Consider the following non-linear stochastic differential equation
\begin{equation}\label{dx}
\rmd x(t)=Ax(t)\rmd t+f(t,x(t))\rmd t+g(t,x(t))\rmd W(t),~~~~t \in
\mathbb{R},
\end{equation}
where $f: \mathbb{R}\times \mathcal{L}^2(\mathbf{P}, \mathbb{H})
\rightarrow \mathcal{L}^2(\mathbf{P}, \mathbb{H})$ and $g:
\mathbb{R}\times \mathcal{L}^2(\mathbf{P}, \mathbb{H}) \rightarrow
\mathcal{L}^2(\mathbf{P}, \mathbb{H})$, and $W(t)$ is a two-sided
standard one-dimensional Brown motion defined on the filtered
probability space $(\Omega, \mathcal{F},\mathbf{P},\mathcal{F}_t)$,
where $\mathcal F_t=\sigma\{W(u)-W(v);u,v\le t\}$.

As in previous section, we also assume that $A$ generates a
$\mathcal{C}_0$-semigroup $(T(t)_{t\geq 0})$, such that
\begin{equation}\label{ed}
\|T(t)\| \leq Ke^{-\omega t},~~~~\mbox{for all}~~~ t \geq 0
\end{equation}
with $K>0$, $\omega>0$.

\begin{de}\label{def}
An $\mathcal F_t$-progressively measurable stochastic process
$\{x(t)\}_{t\in \mathbb{R}}$ is called a {\em mild solution} of
(\ref{dx}) if it satisfies the corresponding stochastic integral
equation
\begin{eqnarray*}
x(t)=T(t-r)x(r)+\int_r^tT(t-s)f(s,x(s))\rmd
s+\int_r^tT(t-s)g(s,x(s))\rmd W(s),
\end{eqnarray*}
for all $t\ge r$ and each $r\in\mathbb{R}$.
\end{de}

\begin{thm}\label{thnon}
Assume $f$ and $g$  are square-mean almost automorphic processes in
$t\in \mathbb{R}$ for each $x\in \mathcal{L}^2(\mathbf{P},
\mathbb{H})$. Moreover $f$ and $g$ satisfy Lipschitz conditions in
$x$ uniformly for $t$, that is, for all $x,y\in
\mathcal{L}^2(\mathbf{P},\mathbb{H})$ and $t\in \mathbb{R}$,
\begin{equation*}
\mathbf{E}\|f(t,x)-f(t,y)\|^2\leq L\mathbf{E}\|x-y\|^2,
\end{equation*}
\begin{equation*}
\mathbf{E}\|g(t,x)-g(t,y)\|^2\leq L'\mathbf{E}\|x-y\|^2,
\end{equation*}
for constants $L,L'>0$. Then (\ref{dx}) has a unique square-mean
almost automorphic mild solution, provided
$\frac{2K^2L}{\omega^2}+\frac{K^2L'}{\omega}<1$.
\end{thm}
\begin{proof}
By definition \ref{def}, stochastic process $x:\mathbb
R\to\mathcal{L}^2(\mathbf{P},\mathbb{H})$ is a solution to
\eqref{dx} if and only if it satisfies the stochastic integral
equation
\begin{eqnarray*}
x(t)=T(t-r)x(r)+\int_r^tT(t-s)f(s,x(s))\rmd
s+\int_r^tT(t-s)g(s,x(s))\rmd W(s).
\end{eqnarray*}
If we let $r\to-\infty$ in above integral equation, by the
exponential dissipation condition of $T$, \eqref{ed}, then we obtain
that the stochastic process $x$ is a solution to \eqref{dx} if and
only if $x$ satisfies the stochastic integral equation
\begin{eqnarray*}
x(t)=\int_{-\infty}^tT(t-s)f(s,x(s))\rmd
s+\int_{-\infty}^tT(t-s)g(s,x(s))\rmd W(s).
\end{eqnarray*}

To seek the square-mean almost automorphic mild solution, let us
consider the nonlinear operator $\mathcal S$ acting on the Banach
space $AA(\mathbb R;\mathcal{L}^2(\mathbf{P},\mathbb{H}))$ given by
\begin{eqnarray*}
(\mathcal{S}x)(t):=\int_{-\infty}^tT(t-s)f(s,x(s))\rmd
s+\int_{-\infty}^tT(t-s)g(s,x(s))\rmd W(s).
\end{eqnarray*}
If we can show that the operator $\mathcal S$ maps $AA(\mathbb
R;\mathcal{L}^2(\mathbf{P},\mathbb{H}))$ into itself and it is a
contraction mapping, then by Banach fixed point theorem, we can
conclude that there is a unique square-mean almost automorphic mild
solution to the equation \eqref{dx}.

Let us consider the nonlinear operators $\mathcal S_1 x$ and
$\mathcal S_2 x$ acting on the Banach space $AA(\mathbb
R;\mathcal{L}^2(\mathbf{P},\mathbb{H}))$ given by
\begin{eqnarray*}
(\mathcal S_1 x)(t):=\int_{-\infty}^t T(t-s)f(s,x(s))\rmd s
\end{eqnarray*}
and
\begin{eqnarray*}
(\mathcal S_2 x)(t):=\int_{-\infty}^t T(t-s)g(s,x(s))\rmd W(s),
\end{eqnarray*}
respectively. By theorem \ref{th}, $F_1(t):=f(t,x(t))$ and
$F_2(t):=g(t,x(t))$ are square-mean almost automorphic if $x$ is,
then by the proof of theorem \ref{th1}, we know that $\mathcal S_1
x$ and $\mathcal S_2 x$ are square-mean almost automorphic if $F_1$
and $F_2$ are. That is, the operator $\mathcal S$ maps $AA(\mathbb
R;\mathcal{L}^2(\mathbf{P},\mathbb{H}))$ into itself.

Next we show that $\mathcal S$ is a contraction mapping on
$AA(\mathbb R;\mathcal{L}^2(\mathbf{P},\mathbb{H}))$. For
$x_1,x_2\in AA(\mathbb R;\mathcal{L}^2(\mathbf{P},\mathbb{H}))$, and
each $t\in\R$ we have
\begin{eqnarray*}
\mathbf{E}\|(\mathcal{S}x_1)(t)-(\mathcal{S}x_2)(t)\|^2&=&\mathbf{E}\|\int_{-\infty}^tT(t-s)[f(s,x_1(s))-f(s,x_2(s))]\rmd s\\
&&+\int_{-\infty}^tT(t-s)[g(s,x_1(s))-g(s,x_2(s))]\rmd W(s)\|^2\\
&\leq& 2K^2\mathbf{E}\left(\int_{-\infty}^te^{-\omega(t-s)}\|f(s,x_1(s))-f(s,x_2(s))\|\rmd s\right)^2\\
&&+2\mathbf{E}\left\|\int_{-\infty}^tT(t-s)[g(s,x_1(s))-g(s,x_2(s))]\rmd
W(s)\right\|^2.
\end{eqnarray*}
We first evaluate the first term of the right-hand side by
Cauchy-Schwarz inequality  as follows:
\begin{eqnarray*}
&&\mathbf{E}\left(\int_{-\infty}^te^{-\omega(t-s)}\|f(s,x_1(s))-f(s,x_2(s))\|\rmd s\right)^2\\
&=&\mathbf{E}\left(\int_{-\infty}^t(e^{\frac{-\omega(t-s)}{2}})(e^{\frac{-\omega(t-s)}{2}})\|f(s,x_1(s))-f(s,x_2(s))\|\rmd s\right)^2\\
&\leq&\mathbf{E}\left[\left(\int_{-\infty}^te^{-\omega(t-s)}\rmd
s\right)\left(\int_{-\infty}^te^{-\omega(t-s)}\|f(s,x_1(s))
       -f(s,x_2(s))\|^2\rmd s\right)\right]\\
&\leq& \left(\int_{-\infty}^te^{-\omega(t-s)}\rmd s\right)\left(\int_{-\infty}^te^{-\omega(t-s)}\mathbf{E}\|f(s,x_1(s))-f(s,x_2(s))\|^2\rmd s\right)\\
&\leq& L\cdot\left(\int_{-\infty}^te^{-\omega(t-s)}\rmd s\right)\left(\int_{-\infty}^te^{-\omega(t-s)}\mathbf{E}\|x_1(s)-x_2(s)\|^2\rmd s\right)\\
&\leq& L\cdot\left(\int_{-\infty}^te^{-\omega(t-s)}\rmd
s\right)^2\sup_{s\in
\mathbb{R}}\mathbf{E}\|x_1(s)-x_2(s)\|^2\\
&\leq& \frac{L}{\omega^2}\cdot\sup_{s\in
\mathbb{R}}\mathbf{E}\|x_1(s)-x_2(s)\|^2.
\end{eqnarray*}
As to the second term, by the Ito's isometry property of stochastic
integral, we have
\begin{eqnarray*}
&&\mathbf{E}\left\|\int_{-\infty}^tT(t-s)[g(s,x_1(s))-g(s,x_2(s))]\rmd W(s)\right\|^2\\
&=&\mathbf{E}\left[\int_{-\infty}^t\|T(t-s)[g(s,x_1(s))-g(s,x_2(s))]\|^2\rmd s\right]\\
&\leq& \mathbf{E}\left[\int_{-\infty}^t\|T(t-s)\|^2\|g(s,x_1(s))-g(s,x_2(s))\|^2\rmd s\right]\\
&\leq& K^2\int_{-\infty}^te^{-2\omega(t-s)}\mathbf{E}\|g(s,x_1(s))-g(s,x_2(s))\|^2\rmd s\\
&\leq& K^2L'\cdot\left(\int_{-\infty}^te^{-2\omega(t-s)}\rmd
s\right)\sup_{s\in
\mathbb{R}}\mathbf{E}\|x_1(s)-x_2(s)\|^2\\
&\leq& \frac{K^2L'}{2\omega}\cdot\sup_{s\in
\mathbb{R}}\mathbf{E}\|x_1(s)-x_2(s)\|^2.
\end{eqnarray*}
Thus, it follows that, for each $t\in\R$,
\begin{eqnarray*}
\mathbf{E}\|(\mathcal{S}x_1)(t)-(\mathcal{S}x_2)(t)\|^2\leq
(\frac{2K^2L}{\omega^2}+\frac{K^2L'}{\omega})\cdot\sup_{s\in
\mathbb{R}}\mathbf{E}\|x_1(s)-x_2(s)\|^2,
\end{eqnarray*}
that is,
\begin{equation}\label{equ1}
\|(\mathcal{S}x_1)(t)-(\mathcal{S}x_2)(t)\|_2^2\le\eta\cdot\sup_{s\in
\mathbb{R}}\|x_1(s)-x_2(s)\|_2^2
\end{equation}
with $\eta:=\frac{2K^2L}{\omega^2}+\frac{K^2L'}{\omega}$. Note that
\begin{equation}\label{equ2}
\sup_{s\in \mathbb{R}}\|x_1(s)-x_2(s)\|_2^2\le (\sup_{s\in
\mathbb{R}}\|x_1(s)-x_2(s)\|_2)^2,
\end{equation}
and \eqref{equ1} together with \eqref{equ2} gives, for each
$t\in\mathbb R$,
\[
\|\mathcal S(x_1)(t)-\mathcal
S(x_2)(t)\|_2\le\sqrt{\eta}~\|x_1-x_2\|_{\infty}.
\]
Hence
\[
\|\mathcal{S}x_1-\mathcal{S}x_2\|_{\infty}=\sup_{t\in
\mathbb{R}}\|\mathcal S(x_1)(t)-\mathcal
S(x_2)(t)\|_2\le\sqrt{\eta}~\|x_1-x_2\|_{\infty}.
\]
Since $\eta<1$, it follows that $\mathcal{S}$ is a contraction
mapping on $AA(\mathbb R;\mathcal{L}^2(\mathbf{P},\mathbb{H}))$.
Therefore, there exists a unique $v\in AA(\mathbb
R;\mathcal{L}^2(\mathbf{P},\mathbb{H}))$ such that $\mathcal{S}v=v$,
which is the unique solution to \eqref{dx}. The proof is now
complete.
\end{proof}

\begin{rem} If $f$ and $g$ in \eqref{dt} and \eqref{dx} are almost periodic in $t$, then the unique
square-mean almost automorphic solution obtained in theorems
\ref{th1} and \ref{thnon} is actually almost periodic, see
\cite{AT,BD,BD1,Pra95,Hal,Tud}.
\end{rem}

\section{Stability of the unique square-mean almost automorphic solution}
\setcounter{equation}{0}

In previous section, for the non-linear stochastic differential
equation \eqref{dx}, we obtain that it has a unique square-mean
almost automorphic solution. In this section, we will show that the
unique square-mean almost automorphic solution is asymptotically
stable in square-mean sense and that any other solutions converge to
it exponentially fast. Firstly, let us state the definition of
asymptotic stability.

\begin{de}\label{sta}
The unique square-mean almost automorphic solution $x_{aa}(t)$ of
(\ref{dx}) is said to be {\em stable in square-mean sense}, if for
arbitrary $\epsilon>0$, there exists $\delta>0$ such that
\begin{eqnarray*}
\mathbf{E}\|x_c(t)-x_{aa}(t)\|^2<\epsilon, \quad t\ge 0
\end{eqnarray*}
whenever $\|c-x_{aa}(0)\|^2<\delta$, where $x_c(t)$ stands for the
solution of (\ref{dx}) with initial condition $x_c(0)=c$. The
solution $x_{aa}(t)$ is said to be {\em asymptotically stable in
square-mean sense} if it is stable in square-mean sense and
\begin{equation*}
\lim_{t\to \infty}\mathbf{E}\|x_c(t)-x_{aa}(t)\|^2=0.
\end{equation*}
\end{de}

\begin{thm}
Assume that the assumptions of theorem \ref{thnon} hold, then the
unique square-mean almost automorphic solution $x_{aa}(t)$ of
\eqref{dx} is asymptotically stable in square-mean sense.
\end{thm}

\begin{proof} We actually will prove more general result. Assume
that $x(t)$ and $y(t)$ are two solutions of \eqref{dx} with initial
values $x(0)$ and $y(0)$, respectively. Note that, by the
exponential dissipation of $T(t)$, we have
\begin{eqnarray}
\|y(t)-x(t)\|^2&=&\|T(t)[y(0)-x(0)]+\int_0^tT(t-s)[f(s,y(s))-f(s,x(s))]\rmd s\nonumber\\
&&+\int_0^tT(t-s)[g(s,y(s))-g(s,x(s))]\rmd W(s)\|^2\nonumber\\
&\leq& 3K^2e^{-2\omega t}\|y(0)-x(0)\|^2+3\left[\int_0^tKe^{-\omega(t-s)}\|f(s,y(s))-f(s,x(s))\|\rmd s\right]^2\nonumber\\
&&+3\left[\int_0^tKe^{-\omega(t-s)}[g(s,y(s))-g(s,x(s))]\rmd
W(s)\right]^2\label{e1}
\end{eqnarray}
Since $f$ satisfies Lipschitz condition, using Cauchy-Schwarz
inequality, we have
\begin{eqnarray}
&&3\left[\int_0^tKe^{-\omega(t-s)}\|f(s,y(s))-f(s,x(s))\|\rmd s\right]^2\nonumber\\
&\leq& 3K^2L^2\left(\int_0^te^{-\omega(t-s)}\|y(s)-x(s)\|\rmd s\right)^2\nonumber\\
&=& 3K^2L^2\left(\int_0^t (e^{\frac{-\omega(t-s)}{2}})(e^{\frac{-\omega(t-s)}{2}})\|y(s)-x(s)\|\rmd s\right)^2\nonumber\\
&\leq& 3K^2L^2(\int_0^te^{-\omega(t-s)}ds)(\int_0^te^{-\omega(t-s)}\|y(s)-x(s)\|^2ds)\nonumber\\
&\leq&
\frac{3K^2L^2}{\omega}\int_0^te^{-\omega(t-s)}\|y(s)-x(s)\|^2\rmd
s,\label{e2}
\end{eqnarray}
for all $t \geq 0$. By Ito's isometry property of stochastic
integral, it follows that
\begin{eqnarray}
&&3\left[\int_0^tKe^{-\omega(t-s)}[g(s,y(s))-g(s,x(s))]\rmd
W(s)\right]^2\nonumber\\
&\leq& 3K^2L'^2\int_0^te^{-2\omega(t-s)}\|y(s)-x(s)\|^2\rmd
s\label{e3}
\end{eqnarray}
for all $t \geq 0$.

Let $Y(t):=\mathbf{E}\|y(t)-x(t)\|^2$ and $k:=3K^2\hat
L^2(1+\frac{1}{\omega})$ with $\hat L:=\max\{L,L'\}$. Note that
$e^{-2\omega t} \leq e^{-\omega t} $ for $t \geq 0$ and by
\eqref{e1}, \eqref{e2}, \eqref{e3}, we have
\begin{eqnarray}\label{yt}
Y(t)\leq 3K^2e^{-\omega t}Y(0)+ k\int_0^t e^{-\omega (t-s)}Y(s)\rmd
s.
\end{eqnarray}
The $Y(t)$ in inequality \eqref{yt} can be controlled by $\widetilde
Y(t)$, which satisfies
\begin{eqnarray*}
\dot{\widetilde Y}(t)&=& -\omega\widetilde Y(t) + k\widetilde Y(t)\\
\widetilde Y(0)&=&3K^2Y(0).
\end{eqnarray*}
Hence $\widetilde Y(t)\rightarrow 0$ exponentially fast if $-\omega
+k<0$, that is
\begin{equation}\label{yt2}
\omega >3K^2\hat L^2(1+\frac{1}{\omega}).
\end{equation}
Note that \eqref{yt2} holds if and only if
\begin{equation*}
\omega^2>3K^2\hat L^2(\omega+1),
\end{equation*}
which always holds. Therefore, $Y(t)$ converges to $0$ exponentially
fast.

In particular, if we set $x(t)=x_{aa}(t)$ in above arguments, we
obtain that the unique square-mean almost automorphic solution
$x_{aa}(t)$ of \eqref{dx} is asymptotically stable in square-mean
sense. The proof is complete.
\end{proof}

\begin{rem}
By the proof of the stability theorem, we actually obtain that any
solution of \eqref{dx} (and hence \eqref{dt}) is asymptotically
stable in square-mean sense. 
\end{rem}

\noindent{\bf Acknowledgment:} The authors are indebted to professor
Yong Li for his encouragement and helpful discussions. The authors
would like to thank the anonymous referee for his/her very careful
reading the manuscript and valuable suggestions which greatly
improved the paper.

\end{document}